\documentclass[a4paper,11pt]{amsart}

\usepackage{amssymb}
\usepackage{amsmath}
\usepackage{amsthm,color}

\allowdisplaybreaks

\usepackage{amsfonts}
\usepackage{amssymb}
\usepackage{amsmath}
\usepackage{amsthm}
\usepackage{bbm}
\usepackage{verbatim}
\usepackage{url}
\usepackage[latin1]{inputenc} %Unix merkist??¦¸
\usepackage[T1]{fontenc}
\usepackage{amsfonts}
\usepackage{amsmath,amssymb,amsthm}
\usepackage{mathrsfs}
\usepackage{verbatim}

\usepackage{graphicx}
\usepackage{ifpdf}

\newtheorem{thm}{Theorem}[section]

\newtheorem{lem}[thm]{Lemma}

\theoremstyle{definition} %makes theorem-like environments
\newtheorem{defn}[thm]{Definition}

\newcommand{\pa}{\varphi^t}

\newcommand{\R}{\mathbb{R}}

\numberwithin{equation}{section}

\def\rr{{\mathbb R}}

\def\zz{{\mathbb Z}}
\def\nn{{\mathbb N}}

\def\cm{{\mathcal M}}

\def\fz{\infty}
\def\az{\alpha}
\def\supp{{\mathop\mathrm{\,supp\,}}}

\def\lz{\lambda}

\def\vz{\varphi}

\def\pa{\partial}
\def\wz{\widetilde}

\def\prz{{\partial}}
\def\gratx{{\nabla_{t,\,x}}}

\def\plz{{P^{[\lz]}_t}}

\def\prz{{\partial}}

\def\ls{\lesssim}
\def\gs{\gtrsim}

\def\tbz{{\triangle_\lz}}
\def\dmz{{dm_\lz}}

\def\riz{{R_{\Delta_\lz}}}

\def\plz{{P^{[\lz]}_t}}

\def\lpz{{L^p(\rr_+,\, dm_\lz)}}

\def\dfrac{\displaystyle\frac}

\def\r{\right}
\def\lf{\left}

\def\noz{\nonumber}

\def\XXint#1#2#3{{\setbox0=\hbox{$#1{#2#3}{\int}$}
     \vcenter{\hbox{$#2#3$}}\kern-.5\wd0}}

\allowdisplaybreaks
\begin{document}

\title{A Moser type inequality for Bessel Laplace equations and applications}

\author{Xuan Thinh Duong}
\address{Xuan Thinh Duong, Department of Mathematics\\
         Macquarie University\\
         NSW 2019\\
         Australia
         }
\email{xuan.duong@mq.edu.au}

\author{Zihua Guo}
\address{Zihua Guo, School of Mathematical Sciences\\
         Monash University\\
         VIC, Australia
         }
\email{zihua.guo@monash.edu}

\author{Ji Li}
\address{Ji Li, Department of Mathematics\\
         Macquarie University\\
         NSW 2019\\
         Australia
         }
\email{ji.li@mq.edu.au}

%
%\author{Yumeng Ou}
%\address{Department of Mathematics\\
%         Brown University\\
%         Providence, RI 02912, USA
%         }
%\email{yumeng.ou@brown.edu}
%

%\author{Brett D. Wick}
%\address{Brett D. Wick, Department of Mathematics\\
%         Washington University -- St. Louis\\
%         St. Louis, MO 63130-4899 USA
%         }
%\email{wick@math.wustl.edu}

\author{Dongyong Yang}
\address{Dongyong Yang, School of Mathematical Sciences\\
 Xiamen University\\
  Xiamen 361005,  China
  }
\email{dyyang@xmu.edu.cn }

%\thanks{}

%\thanks{Date: \today}

%    General info
\subjclass[2010]{31C05, 35J05, 42B25, 42B30}
% 42B35  Function spaces arising in harmonic analysis
% 42C40  Wavelets and other special systems
% 30L99  Analysis on metric spaces: None of the above,
%        but in this section
% 42B30  $H^p$-spaces
% 42B25  Maximal functions, Littlewood-Paley theory

%\date{May 28, 2013}
\date{\today}

%\dedicatory{This paper is dedicated to\ldots}

\keywords{Bessel operator, harmonic function, Moser inequality}

\begin{abstract}
In this paper, we study Bessel operators and Bessel Laplace equations studied by Weinstein, Huber, and related the harmonic function theory introduced by Muckenhoupt--Stein.  We establish the Moser type inequality for these harmonic functions, which is missing in this setting before. We then apply it to give a direct proof  for the equivalence of characterizations of the Hardy spaces associated to Bessel operator via non-tangential maximal function and  radial maximal function defined in terms of the Poisson semigroup.
\end{abstract}

\maketitle

%\tableofcontents

%%%%%%%%%%%%%%%%%
% INTRODUCTION
%%%%%%%%%%%%%%%%%

%----------------------------------------------------------

\section{Introduction and statement of main results}
\label{sec:introduction}
\setcounter{equation}{0}

%\color{red}
%Recall the results on Euclidean setting: (for Brett to finalise)
%\color{black}

The study of harmonic functions in different spaces and domains has played an important role in harmonic analysis and partial differential equations. For  harmonic functions, one can list the maximum principle, Harnack's inequality among their important properties. Another useful feature is the so-called  Moser inequality that we remind the reader now.
Suppose $u(t,x)$ is a harmonic function on $\mathbb{R}_+^{n+1}$, i.e.,
$$\Big( {\pa^2\over\pa t^2} + {\pa^2\over\pa x_1^2}+\cdots+ {\pa^2\over\pa x_n^2}\Big) u(t,x) =0,\quad x:=(x_1,\ldots,x_n)\in\mathbb{R}^n, t>0,$$
then we have that for any $0<p<\infty$, there exists a positive constant $C_{n,p}$ depending only on $n$ and $p$ such that
$$ |u(t_0,x_0)|\leq C_{n,p} \bigg({1\over |B|} \int_B |u(t,x)|^pdtdx \bigg)^{1\over p}  $$
for any $(t_0,x_0)\in \mathbb{R}_+^{n+1}$ and any ball $B\subset \mathbb{R}_+^{n+1}$ centered at $(t_0,x_0)$.
C. Fefferman and E. Stein  \cite[Section 9, Lemma 2]{FS} first proved this inequality for harmonic functions, using the Poisson representation on the sphere (note that they also pointed out that this result is essentially due to Hardy and Littlewood \cite{HaLi}). Later, Han and Lin \cite{hl97} reproved this Moser inequality on $\mathbb{R}^n$ for $n\geq 3$, where they used the properties of the fundamental solution of the Laplace equation and the construction of suitable test functions to obtain a version of this type of inequality for $p=2$, and then, by the standard iteration approach, they proved it for general $p\in (0,\infty)$.

In 1965, B. Muckenhoupt and E. Stein in \cite{ms} introduced the harmonic function theory associated with Bessel operator $\tbz$, defined by
setting for suitable functions $f$,
\begin{equation*}
\tbz f(x):=\frac{d^2}{dx^2}f(x)+\frac{2\lz}{x}\frac{d}{dx}f(x),\quad \lz>0,\quad x\in \R_+:=(0,\fz).
\end{equation*}
The associated Bessel Laplace equation given by
\begin{equation}\label{bessel laplace equation}
\triangle_{t,\,x} (u) :=\pa_{t}^2u + \pa_{x}^2u+\frac{2\lz}{x}\pa_{x}u=0
\end{equation}
was studied by A. Weinstein \cite{w}, and A. Huber \cite{Hu} in higher dimension. In these works, they considered the generalised auxially symmetric potentials, and obtained the properties of the solutions of this equation, such as the extension, the uniqueness theorem, and the boundary value problem for certain domains.

If $u\in C^2(\R_+\times\R_+)$ is a solution of \eqref{bessel laplace equation} then $u$ is said to be $\lambda$-harmonic. The function $u$ and its conjugate (denoted by $v$) satisfy the following Cauchy--Riemann type equations
\begin{align}\label{CR}
\pa_{x}u=-\pa_{t}v\ \ {\rm and\ \ }
\pa_{t}u =\pa_{x}v + {2\lambda\over x} v\ \ {\rm in\ \ }\mathbb{R}_+\times \R_+.
\end{align}

In \cite{ms} they developed a theory of functions in the setting of
$\tbz$ which parallels the classical one associated to the standard Laplacian, where results on $\lpz$-boundedness of conjugate
functions and fractional integrals associated with $\tbz$ were
obtained for $p\in[1, \fz)$ and $\dmz(x):= x^{2\lz}\,dx$.
We also point out that Haimo \cite{h} studied the Hankel convolution transforms $\vz \sharp_\lz f$ associated with the Hankel transform in the Bessel setting systematically,
which provides a parallel theory to the classical convolution and Fourier transforms.  Also note that
the Poisson integral of $f$ studied in \cite{ms} is the Hankel convolution of Poisson kernel with $f$, see \cite{bdt}.

Since then, many problems based on the Bessel context were studied, such as the boundedness of Bessel Riesz transform,
Littlewood--Paley g-functions, Hardy and BMO spaces associated with Bessel operators, $A_p$ weights associated with Bessel operators
(see, for example, \cite{k78,ak,bfbmt,v08,bfs,bhnv,bcfr,yy,dlwy,DLMWY,dlwy2} and the references therein).

%\smallskip
\vskip.15cm

The aim of this paper is to give a positive answer to the open question whether the Moser inequality is true for $\lambda$-harmonic functions.  We will prove the Moser inequality in Section 2, then apply it to show the equivalence of
characterizations of the Hardy spaces associated to Bessel operator via non-tangential maximal function and radial maximal function.

To be more precise, for any $(t_0, x_0)\in (\R_+\cup\{0\})\times(\R_+\cup\{0\})$ and $R>0$, we define the ball $B((t_0,\, x_0),\, R)$ as follows
\begin{align}\label{ball}
B((t_0,\, x_0),\, R):=\lf\{(t, x)\in\R_+\times\R_+: (t-t_0)^2+(x-x_0)^2<R^2\r\}.
\end{align}
We also define the measure of these balls as 
$\tilde m_\lz(B((t_0,\, x_0),\, R)):= \int_{B((t_0,\, x_0),\, R)} 1dt x^{2\lambda} dx.$

Our main result  is the following.
\begin{thm}\label{t-moser inequ} Suppose $u\in C^2(\R_+\times\R_+)$ and $u$ is a solution of \eqref{bessel laplace equation}.
Let $p\in(0, \fz)$. Then there exists a positive constant $C_{p,\,\lz}$ such that for any $(t_0, x_0)\in\R_+\times\R_+$ and $R>0$, we have
\begin{align}\label{sub har inequ}
\sup_{(t,\,x)\in B((t_0,\,x_0),\,R)}| u(t, x)|\le \lf[\dfrac{ C_{p,\,\lz}}{\tilde m_\lz(B((t_0, x_0), 12R))}\iint_{B((t_0,\, x_0),\, 12R)}| u(t, x)|^px^{2\lz}dx\,dt\r]^{1/p}.
\end{align}
\end{thm}

As we mentioned earlier, in the classical case, there are two different approaches to prove the Moser inequality for harmonic functions: (i) via Poisson representation (\cite{FS}), (ii) via fundamental solutions of the Laplace equation and the iteration (\cite{hl97}).  

However, we point out that in this Bessel setting, none of these two methods can be applied directly. One of the main difficulties here is that for a given ball $B((t_0, x_0), R)$,
the measure $\tilde m_\lz(B((t_0, x_0),R))$ depends on both its radius and center, i.e., $\tilde m_\lz$ is not translation invariant with respect to $x$.
Thus, for the first approach, note that for the Poisson representation for the $\lambda$-harmonic functions $u$ (i.e., $u$ satisfies \eqref{bessel laplace equation}), it is only known for the points $(0,0)$ or $(t,0)$, but not known when $x\not=0$. Hence, this approach is not applicable when we consider the points $(t,x)$ with $x>t$. For the second approach, we first point out that the ``homogeneous dimension'' with respect to the measure $\tilde m_\lz$ is $2\lz+2$, which is not the standard one. Moreover, the idea of constructing some suitable Schwarz functions does not work in this setting since the measure $\tilde m_\lz$ depends on both the radius and the center of the ball.

To prove our main theorem, we need to combine some of the ideas of these two approaches, together with our new observation of applying the Sobolev embedding theorems. To be more precise, we observe that the equation \eqref{bessel laplace equation} is translation invariant with respect to $t$ but not $x$. Thus, it is natural to compare the radius $R$ with the coordinate $x_0$ of the centre. We will consider the following two cases:  Case (i): $R\leq x_0/4$, and Case (ii): $R> x_0/4$.

To handle Case (i), we first establish the Caccioppoli inequality for the $\lambda$-harmonic function $u$, then establish a suitable version of 
Sobolev embedding theorem in this setting to obtain the required estimates for the partial derivatives of $u$. Combining these estimates together, 
we obtain an $L^2$ version of the Moser type inequality for  $u$, which implies \eqref{sub har inequ} by using the standard iteration approach.

To prove Case (ii), noting that in this case, $R> x_0/4$, we can consider a larger ball centered at $(t_0,0)$ with radius $5R$, which contains the ball $B((t_0,x_0),R)$ as defined in \eqref{ball} and with comparable measures. Then we apply the Poisson representation for the $\lambda$-harmonic functions, and follows the idea of  C. Fefferman and E. Stein  \cite[Section 9, Lemma 2]{FS} to obtain  \eqref{sub har inequ}.

%r\quad

As one of the applications, we can prove directly that the $L^p$ norms of the non-tangential maximal function and radial maximal function of the Poisson integral of $f$ are equivalent. For more details, we refer the reader to Section 3.

Throughout the paper, for every interval $I\subset \R_+$, we denote it by $I:=I(x,t):= (x-t,x+t)\cap \R_+$. The measure of $I$ is defined as
$m_\lz(I(x,t)):=\int_{I(x,\,t)} x^{2\lz} dx$.

\section{Proof of Theorem \ref{t-moser inequ} }
\label{sec:proof}
\setcounter{equation}{0}

To begin with, we consider two cases:\ \   Case (i): $R\leq x_0/4$, and Case (ii): $R> x_0/4$.

We now consider Case (i). We claim that 
when $R\leq x_0/4$, we have
\begin{align}\label{sub har inequ 1}
\sup_{(t,\,x)\in B((t_0,\,x_0),\,R)}| u(t, x)|\le \lf[\dfrac{ C_{p,\,\lz}}{\tilde m_\lz(B((t_0, x_0), 2R))}\iint_{B((t_0,\, x_0),\, 2R)}| u(t, x)|^px^{2\lz}dx\,dt\r]^{1/p}.
\end{align}

Note that in this case, we have $2R\leq x_0/2$, then for $(t,x)\in B((t_0,x_0), 2R)$ we have $x\sim x_0$ and thus \[\tilde m_\lz(B((t_0, x_0), 2R))\sim x_0^{2\lambda}R^2.\] To handle this case, we will establish the Caccioppoli inequality, Sobolev embedding theorem to obtain an $L^2$ version of the Moser type inequality for $p=2$, and then using an iteration approach (see for example \cite{hl97}) to obtain the general case for $p\in(0,\infty)$.

We first  establish the Caccioppoli inequality in this Bessel setting as follows.

\begin{lem}[Caccioppoli inequality]\label{l-cacci lemma}
Let $(t_0, x_0)\in\R_+\times\R_+$ and $R\in(0, \fz)$. Then there exists a
positive constant $C$, independent of $(t_0,x_0)$,
$R$ and $ u$, such that
\begin{equation*}
\iint_{B((t_0,\,x_0),\,R)}|\gratx  u(t,x)|^2
\,\dmz(x)\,dt\le \frac C{R^2}
\iint_{B((t_0,\,x_0),\,2R)}|u(t,x)|^2\,\dmz(x)\,dt.
\end{equation*}
\end{lem}

\begin{proof}
Choose a function $\eta\in C^\fz_c(\R_+\times\R_+)$ with $\supp(\eta)\subset
B((t_0,x_0), 2R)$, $0\le\eta\le1$,
$\eta\equiv1$ on $B((t_0,x_0), R)$ and
$|\gratx \nabla\eta|\ls R^{-1}$. Define $\vz:=\eta^2 u$.  By \eqref{bessel laplace equation}, we see that
\begin{eqnarray*}
0&&= \iint_{B((t_0,\,x_0),\,2R)}\lf[\prz_t^2  u(t,x)+\prz_x^2  u(t,x)+\frac{2\lz}x\prz_x  u(t,x)\r]\overline{\varphi(t,x)}x^{2\lz}\,dx\,dt\\
&&= \iint_{B((t_0,\,x_0),\,2R)}\lf[\prz_x^2  u(t,x)+\frac{2\lz}x\prz_x  u(t,x)\r]\overline{\eta^2(t,x) u(t,x)}x^{2\lz}\,dx\,dt\\
&&\quad+ \iint_{B((t_0,\,x_0),\,2R)}\prz_t^2  u(t,x)\overline{\eta^2(t,x) u(t,x)}x^{2\lz}\,dx\,dt\\
&&=-\iint_{B((t_0,\,x_0),\,2R)}\prz_x  u(t,x)\overline{\prz_x  u(t,x)\eta^2+ 2 u(t,x)\eta\prz_x \eta} x^{2\lz}\,dx\,dt\\
&&\quad- \iint_{B((t_0,\,x_0),\,2R)}\prz_t  u(t,x)\overline{\prz_t  u(t,x)\eta^2+ 2 u(t,x)\eta\prz_t \eta}x^{2\lz}\,dx\,dt.
\end{eqnarray*}
Applying the Cauchy-Schwarz inequality to each term in the right-hand side above, and then adding them up, we obtain  that
\begin{eqnarray*}
 &&\iint_{B((t_0,\,x_0),\,2R)} |\gratx  u(t,x)|^2\eta^2\,x^{2\lz}\,dx\,dt\\
 &&\quad\le2\iint_{B((t_0,\,x_0),\,2R)} \lf[\frac14|\gratx  u(t,x)|^2\eta^2+4| u(t,x)|^2\lf|\gratx \eta\r|^2\r]\,x^{2\lz}\,dx\,dt.
\end{eqnarray*}
Combining this with the property $\eta\equiv1$ on $B((t_0,x_0), R)$, we further deduce that
\begin{eqnarray*}
 \iint_{B((t_0,\,x_0),\,R)} |\gratx  u(t,x)|^2\,x^{2\lz}\,dx\,dt
&\ls& \iint_{B((t_0,\,x_0),\,2R)}
| u(t,x)|^2\lf|\gratx \eta\r|^2\,x^{2\lz}\,dx\,dt\\
&\ls&\frac1{R^2}\iint_{B((t_0,x_0),\, 2R)} | u(t,x)|^2\,x^{2\lz}\,dx\,dt.
\end{eqnarray*}
This finishes the proof of Lemma \ref{l-cacci lemma}.
\end{proof}

\begin{lem}[Sobolev Embedding]\label{embedding lemma}
Let $(t_0, x_0)\in\R_+\times\R_+$ and $R\in(0, \fz)$. Then for the ball $B((t_0, x_0), R)$, and for $f\in C^2(\R+\times\R_+)$ we have
\begin{align*}
\|f\|_{L^\infty(B((t_0, x_0), R))}&\ls \Big( {1\over R^2} \int_{B((t_0, x_0), R))} |f(t,x)|^2dtdx \Big)^{1\over 2}+ \Big(  \int_{B((t_0, x_0), R))} |\nabla_{t,x} f(t,x)|^2dtdx \Big)^{1\over 2}\\
&\quad +\Big( R^2 \int_{B((t_0, x_0), R))} |\nabla_{t,x}^2f(t,x)|^2dtdx \Big)^{1\over 2}.
\end{align*}
\end{lem}
%We note that Lemma \ref{embedding lemma} is the standard Sobolev embedding theorem, see for example \cite{}.
\begin{proof}
For $R=1$, this is a consequence of the standard Sobolev embedding, see for example \cite{Adam}. That is,
\begin{align}\label{eee1}
\|f\|_{L^\infty(B((t_0, x_0), 1))}&\ls \Big( \int_{B((t_0, x_0), 1))} |f(t,x)|^2dtdx \Big)^{1\over 2}+ \Big(  \int_{B((t_0, x_0), 1))} |\nabla_{t,x} f(t,x)|^2dtdx \Big)^{1\over 2}\\
&\quad +\Big(  \int_{B((t_0, x_0), 1))} |\nabla_{t,x}^2f(t,x)|^2dtdx \Big)^{1\over 2}.\noz
\end{align}

For general case, we use the rescaling.
Namely, for any $R\in(0, \fz)$, consider the function $f(t,x)$ on the ball $B((t_0, x_0), R)$. Let $\tilde t={1\over R}(t-t_0)$ and
$\tilde x={1\over R}(x-x_0)$. Then it is obvious that $(\tilde t,\tilde x)\in B((t_0, x_0), 1)$. Now define
$$g(\tilde t,\tilde x)=f(R \tilde t+t_0, R\tilde x+x_0).$$ Then, by applying \eqref{eee1} to $g(\tilde t, \tilde x)$ and changing of variables, we obtain our version of 
the embedding result.
\end{proof}

%We now borrow the idea of iteration in \cite{hl97} to establish the Moser-type inequality for the Bessel harmonic function $u(t,x) $ as follows.

%\begin{proof}[Proof of Theorem \ref{t-moser inequ}]

We now apply the Caccioppoli inequality, the Sobolev embedding theorem to prove the Moser type inequality in $L^2$. To be more precise,
we claim that for all $\az\in(0, 1)$,
$r\in (0, 2R]$ and all $(t, x)\in B((t_0,x_0), \az r)$,
\begin{equation}\label{2.4}
| u(t,x)|\ls  \lf\{\frac1{ \wz m_\lambda (B((t_0,x_0),(1-\az)r)) }
\iint_{ B((t_0,\,x_0),\,r)}| u(s,y)|^2\,y^{2\lambda}dy\,ds\r\}^{1/2},
\end{equation}
where the implicant constant is independent of $ u$, $r$, $\az$, $t_0$ and $x_0$.

To see this, consider the ball $\mathbb B:=B((t,x), (1-\alpha)r/2)$.
%
%Case 1: $t_0 < {x_0\over 2}$. In this case, $\az r, (1-\az)r < {x_0\over 4}$.
%
%This implies that for  any $(t,x)\in B((t_0,x_0), \az r)$,
%$(1-\az)r<x/2 $.
%and then for any $(\wz t, y)\in \mathbb B\setminus {1\over2}\mathbb B$,
%$y/2< x<3y/2$.
Observe that $x\sim x_0$, and moreover, for any $(s,y)\in \mathbb B$,  we have  $y\sim x_0$.

In this case, we first note that $\partial_t u$ is also a solution of \eqref{bessel laplace equation} since $u$ is a solution. Hence, applying the Caccioppoli inequality in Lemma \ref{l-cacci lemma} to $\partial_t u$, we get that
\begin{align}\label{ee1}
\iint_{\mathbb B}|\gratx  \partial_t u(t,x)|^2
\,\dmz(x)\,dt  &\le \frac C{(1-\az)^2r^2}
\iint_{\mathbb B}|\partial_t u(t,x)|^2\,\dmz(x)\,dt\\
&\le \frac C{(1-\az)^2r^2}
\iint_{\mathbb B}|\nabla_{t,x} u(t,x)|^2\,\dmz(x)\,dt\noz\\
&\le \frac C{(1-\az)^4r^4}
\iint_{\mathbb B}|u(t,x)|^2\,\dmz(x)\,dt.\noz
\end{align}

Again, since $u$ is  a solution of \eqref{bessel laplace equation}, we have
\begin{align*}
\iint_{\mathbb B}|\partial_x^2 u(t,x)|^2
\,\dmz(x)\,dt 
 &\ls \iint_{\mathbb B}|\partial_t^2 u(t,x)|^2
\,\dmz(x)\,dt+\iint_{\mathbb B}\Big| {2\lz\over x} \partial_x u(t,x)\Big|^2
\,\dmz(x)\,dt \\
&=: I+II.
\end{align*}

Note that the term $I$ is bounded by the left-hand side of \eqref{ee1}. We obtain that
\begin{align}\label{e of I}
I
\le \frac C{(1-\az)^4r^4}
\iint_{\mathbb B}|u(t,x)|^2\,\dmz(x)\,dt.
\end{align}
For the term $II$, using the Caccioppoli inequality in Lemma \ref{l-cacci lemma} and the fact that $(1-\az)r<x/2 $, we get
\begin{align*}
II
&\le \frac C{(1-\az)^4r^4}
\iint_{\mathbb B}|u(t,x)|^2\,\dmz(x)\,dt.\noz
\end{align*}
Combining the estimates of $I$ and $II$, we obtain that
\begin{align}\label{ee2}
\iint_{\mathbb B}|\partial_x^2 u(t,x)|^2
\,\dmz(x)\,dt 
&\le \frac C{(1-\az)^4r^4}
\iint_{\mathbb B}|u(t,x)|^2\,\dmz(x)\,dt.
\end{align}
Then we further have
\begin{align}\label{ee3}
&\hskip-.5cm\iint_{\mathbb B}|\nabla_{t,x}^2 u(t,x)|^2
\,\dmz(x)\,dt \\
&= \iint_{\mathbb B}\big(\pa_{t}^2 u(t,x)+ \pa_{x}^2 u(t,x)\big)^2
\,\dmz(x)\,dt\noz \\ 
&\leq 2\iint_{\mathbb B}\big(\pa_{t}^2 u(t,x)\big)^2
\,\dmz(x)\,dt+ 2\iint_{\mathbb B}\big( \pa_{x}^2 u(t,x)\big)^2
\,\dmz(x)\,dt \noz\\
&\le \frac C{(1-\az)^4r^4}
\iint_{\mathbb B}|u(t,x)|^2\,\dmz(x)\,dt,\noz
\end{align}
where the last inequality follows from \eqref{ee2} and the estimate of the term $I$ in \eqref{e of I}.

Now applying the Sobolev embedding in Lemma \ref{embedding lemma} we have
\begin{align*}
\|u\|_{L^\infty(\mathbb B)}&\ls \Big( {1\over (1-\az)^2r^2} \int_{\mathbb B} |u(t,x)|^2dtdx \Big)^{1\over 2}+ \Big(  \int_{\mathbb B} |\nabla_{t,x} u(t,x)|^2dtdx \Big)^{1\over 2}\\
&\quad +\Big( (1-\az)^2r^2 \int_{\mathbb B} |\nabla_{t,x}^2u(t,x)|^2dtdx \Big)^{1\over 2}\\
&\ls \Big( {1\over (1-\az)^2r^2 x_0^{2\lz}} \int_{\mathbb B} |u(t,x)|^2dt\dmz(x) \Big)^{1\over 2}+ \Big( {1\over x_0^{2\lz}} \int_{\mathbb B} |\nabla_{t,x} u(t,x)|^2dt\dmz(x) \Big)^{1\over 2}\\
&\quad +\Big( (1-\az)^2r^2 {1\over x_0^{2\lz}} \int_{\mathbb B} |\nabla_{t,x}^2u(t,x)|^2dt\dmz(x) \Big)^{1\over 2}\\
&\ls  \Big( {1\over (1-\az)^2r^2 x_0^{2\lz}} \int_{\mathbb B} |u(t,x)|^2dt\dmz(x) \Big)^{1\over 2},
\end{align*}
where the second inequality follows from the fact that $x\sim x_0$, and the last inequality follows from the Caccioppoli inequality in Lemma \ref{l-cacci lemma}
and the inequality \eqref{ee3}.

This implies that the claim  \eqref{2.4} holds in this case.

By H\"older's inequality and \eqref{2.4}, we further obtain that \eqref{2.4} holds with 2 replaced by $p$ for all $p\in[2, \fz)$. Then \eqref{sub har inequ 1} for $p\in[2, \fz)$ follows from \eqref{2.4} with  $\az:=1/2$
and $r:=2R$.

We now use the technique of iteration (see for example \cite{hl97}) to establish the Moser-type inequality for the Bessel harmonic function $u(t,x) $ for $p\in(0,2)$ as follows.
Observe that for all
$a,\,b\in(0, \fz)$ and $q\in(1, \fz)$, $ab<a^q+b^{q'}$,
where $q':=\frac q{q-1}$. Then
by this with $q:=\frac1{1-p/2}$ and \eqref{2.4},
there exists positive constant $C$ such that
for all $r\in (0, 2R]$, $\wz r\in (0, r)$ and almost all $(t,x)\in B((t_0, x_0), \wz r)$,
\begin{eqnarray*}
| u(t,x)|&\ls& \| u\|^{1-\frac p2}_{L^\fz(B((x_0,\,t_0),\,r))}\wz m_\lz( B((t_0,x_0),r-\wz r)  )^{-\frac 12}\| u\|^{\frac p2}_{L^p(B((t_0,\,x_0),\,2R),\,d \tilde m_\lz)}\\
&\le& \frac12\|u\|_{L^\fz(B((t_0,\,x_0),\,r))}+C\wz m_\lz( B((t_0,x_0),r-\wz r)  )^{-\frac1p}\| u\|_{L^p(B((t_0,\,x_0),\,2R),\,d \tilde m_\lz)}.
\end{eqnarray*}

Set $f(\wz r):=\| u\|_{L^\fz(B((t_0,\,x_0),\, \wz r))}$,
where $\wz r\in (0, r]$. By
the last inequality, we see that for all $r,\wz r$ such that $0<\wz r< r\le 2R$,
\begin{equation}\label{2.6}
f(\wz r)\le \frac12 f(r)+ C\wz m_\lz( B((t_0,x_0),r-\wz r)  )^{-\frac1p}\| u\|_{L^p(B((t_0,\,x_0),\,2R),\,d \tilde m_\lz)}.
\end{equation}
Now fix $\wz r$, $r$ and write $r_0:=\wz r$, $r_\fz:= r$, and
$r_{j+1}:= r_{j}+(1-\tau)\tau^{j}(r-\wz r)$ for $j\in\zz_+:=\nn\cup\{0\}$,
where $\tau\in (0, 1)$ such that
$2\tau^{(2\lz+1)/p}>1$.
An application of \eqref{2.6} gives us that
$$f(r_j)\le \frac12f(r_{j+1})+C\wz m_\lz( B((t_0,x_0),r_{j+1}-r_j)  )^{-\frac1p}\| u\|_{L^p(B((t_0,\,x_0),\,2R),\,d \tilde m_\lz)}$$
 for all $j\in\zz_+$. This in turn implies that, by  iteration of $k$ steps,
\begin{eqnarray*}
f(\wz r)&=&f(r_0)\le \frac12f(r_1)
+C\frac{\| u\|_{L^p(B((t_0,\,x_0),\,2R),\,d \tilde m_\lz)}}{\wz m_\lz( B((t_0,x_0),r_{1}-r_0)  )^{\frac1p}}\\
&\le&\lf(\frac12\r)^kf(r_k)
+C\sum_{j=0}^{k-1}2^{-j}\tau^{-j(2\lz+1)/p}
\frac{\| u\|_{L^p(B((t_0,\,x_0),\,2R),\,d\tilde m_\lz)}}{\wz m_\lz( B((t_0,x_0),(1-\tau)(r-\wz r))  )^{\frac1p}}\\
&\le&\lf(\frac12\r)^kf(r_k)
+C\frac{\| u\|_{L^p(B((t_0,\,x_0),\,2R),\,d\tilde m_\lz)}}{\wz m_\lz( B((t_0,x_0),(1-\tau)(r-\wz r))  )^{\frac1p}}.
\end{eqnarray*}
We note that $u\in C^2(\R_+\times\R_+)$, then $u\in L^\fz(B((t_0,\,x_0),\,2R))$. Letting $k\to\fz$ and using $f(r_j)\le f(2R)<\fz$
for all $j\in\zz_+$, we have that
\begin{equation*}
f(\wz r)\ls\frac{\| u\|_{L^p(B((t_0,\,x_0),\,2R),\,d\tilde m_\lz)}}{\wz m_\lz( B((t_0,x_0),(1-\tau)(r-\wz r))  )^{\frac1p}}.
\end{equation*}
Taking $r:=2R$ and $\wz r:=R$, we obtain that
\begin{equation*}
\|u\|_{L^\fz(B((t_0,\,x_0),\,R))}\ls
\frac{\| u\|_{L^p(B((t_0,\,x_0),\,2R),\,d\tilde m_\lz)}}{\wz m_\lz( B((t_0,x_0),R)  )^{\frac1p}}.
\end{equation*}
This finishes the proof of \eqref{sub har inequ 1} for $p\in (0, 2)$
and hence, the proof of case {\rm(i)}. % Theorem \ref{t-moser inequ}.
%\end{proof}

\bigskip
We now consider Case (ii). We claim that 
when $R> x_0/4$, we have
\begin{align}\label{sub har inequ 2}
\sup_{(t,\,x)\in B((t_0,\,x_0),\,R)}| u(t, x)|\le \lf[\dfrac{ C_{p,\,\lz}}{\tilde m_\lz(B((t_0, x_0), 12R))}\iint_{B((t_0,\, x_0),\, 12R)}| u(t, x)|^px^{2\lz}dx\,dt\r]^{1/p}.
\end{align}

Note that in this case, we have $12R\geq 3x_0$, then $B((t_0,x_0),R)\subset B((t_0,0), 5R)\subset B((t_0,x_0),12R)$. We have $\tilde m_\lz(B((t_0, x_0), 12R))\sim R^{2\lambda+2}\sim \tilde m_\lz(B((t_0, 0), 5R))$.  Thus, to prove \eqref{sub har inequ 2}, it suffices to show that 
\begin{align}\label{sub har inequ 2 eee}
\sup_{(t,\,x)\in B((t_0,\,0),R)}| u(t, x)|\le \lf[\dfrac{ C_{p,\,\lz}}{\tilde m_\lz(B((t_0, 0), 5R))}\iint_{B((t_0,\, 0),\, 5R)}| u(t, x)|^px^{2\lz}dx\,dt\r]^{1/p}.
\end{align}
We point out that the equation \eqref{bessel laplace equation} is translation invariant under the variable $t$. Thus, to prove \eqref{sub har inequ 2 eee}, it suffices to prove that
\begin{align}\label{sub har inequ 2 ee1}
\sup_{(t,\,x)\in B((0,\,0),R)}| u(t, x)|\le \lf[\dfrac{ C_{p,\,\lz}}{\tilde m_\lz(B((0, 0), 5R))}\iint_{B((0,\, 0),\, 5R)}| u(t, x)|^px^{2\lz}dx\,dt\r]^{1/p}.
\end{align}

To obtain \eqref{sub har inequ 2 ee1},  we use the Poisson representation of the harmonic function $u$ (\cite[P. 25]{ms}). 
To begin with, following Muckenhoupt and Stein \cite{ms}, we consider the even extension of $u(t,x)$ in terms of $x$, i.e.,
when $x<0$, we define $u(t,x)=u(t,-x)$.

Then we use the polar coordinates to consider the estimates on the circles. Let $x:=r\sin \theta$, $t:=r\cos \theta$, and let 
$$H:=\dfrac{ 1}{\tilde m_\lz(B((0, 0), 5R))}\int_0^{5R}\int_0^\pi| u(r\cos \theta, r\sin \theta)|^p (r\sin \theta)^{2\lz}d\theta\, rdr,$$
$$m_p(r):=\lf[\int_0^\pi| u(r\cos \theta, r\sin \theta)|^p (\sin \theta)^{2\lz}\,d\theta\r]^{1\over p},$$
and
$$m_\infty(r):=\sup_{x^2+t^2=r^2} |u(t, x)|^p, $$
where $r\in (0,5R]$.

To prove \eqref{sub har inequ 2 ee1}, it suffices to prove that there exists $r_0\in (R, 5R)$ such that
\begin{align}\label{sub har inequ 2 ee2}
m_\infty(r_0)\ls H.
\end{align}
In fact, assume \eqref{sub har inequ 2 ee2} at the moment. Then by the maximal principle (Theorem 1 in \cite{ms}), we obtain that for any $(t,x)\in B((0,\,0),\,R)$, we have
$$ |u(t,x)|\leq  m_\infty(r_0)\ls H,$$
which implies \eqref{sub har inequ 2 ee1}.

We now prove \eqref{sub har inequ 2 ee2}. To begin with, we can assume that for all $r\in (R, 5R)$, we have
\begin{align}\label{sub har inequ 2 ee3}
m_\infty(r_0)> H.
\end{align}
since otherwise \eqref{sub har inequ 2 ee2} holds. 

To continue, we first note that
\begin{align}\label{sub har inequ 2 ee4}
m_1(r)&= \int_0^\pi| u(r\cos \theta, r\sin \theta)| (\sin \theta)^{2\lz}\,d\theta\\
&\leq m_\infty(r)^\alpha  \int_0^\pi| u(r\cos \theta, r\sin \theta)|^{1-\alpha} (\sin \theta)^{2\lz}\,d\theta\noz\\
&= m_\infty(r)^\alpha m_p(r)^p,\noz
\end{align}
where $\alpha:=1-p$.

Next we recall the Poisson representation (see \cite[P. 25]{ms}) as follows.
\begin{align}\label{Poisson re}
u(\rho\cos\theta,\rho\sin\theta)=\int_0^\pi P\Big({\rho\over r},\theta,\phi\Big)u(r\cos\phi,r\sin\phi)(\sin \phi)^{2\lz}\,d\phi,
\end{align}
where $\rho<r$, and $P\Big({\rho\over r},\theta,\phi\Big)$ is the Poisson kernel defined as 
\begin{align}\label{Poisson kernel}
P\Big({\rho\over r},\theta,\phi\Big):={\lambda\Big(1-{\rho^2\over r^2}\Big)\over\pi}\int_0^\pi {(\sin\beta)^{2\lz-1} \over  \big[(t-\xi)^2+(x-\eta)^2+2x\eta(1-\cos\beta) \big]^{\lz+1}}d\beta,
\end{align}
with $ t:={\rho\over r}\cos\theta $, $ x:={\rho\over r}\sin\theta $, $ \xi:=\cos\phi $ and $ \eta:=\sin\phi $.

From the definition, it is direct that 
\begin{align*}
\Big\|P\Big({\rho\over r},\theta,\phi\Big)\Big\|_\infty\ls \Big(1-{\rho^2\over r^2}\Big) {1\over  \Big(1-{\rho\over r}\Big)^{2\lz+2}  } \approx {1\over  \Big(1-{\rho\over r}\Big)^{2\lz+1}  }. 
\end{align*}

Then, from \eqref{Poisson re} we obtain that
\begin{align*}
|u(t,x)|&=|u(\rho\cos\theta,\rho\sin\theta)|\\
&\leq \int_0^\pi |u(r\cos\phi,r\sin\phi)|(\sin \phi)^{2\lz}\,d\phi\ \Big\|P\Big({\rho\over r},\theta,\phi\Big)\Big\|_\infty\\
&\ls  {1\over  \Big(1-{\rho\over r}\Big)^{2\lz+1}  } \int_0^\pi |u(r\cos\phi,r\sin\phi)|(\sin \phi)^{2\lz}\,d\phi\\
&\ls {1\over  \Big(1-{\rho\over r}\Big)^{2\lz+1}  } m_1(r),
\end{align*}
which implies that 
\begin{align}\label{sub har inequ 2 ee5}
m_\infty(\rho)\ls {1\over  \Big(1-{\rho\over r}\Big)^{2\lz+1}  } m_1(r).
\end{align}
This, together with \eqref{sub har inequ 2 ee4}, gives
\begin{align*}
m_\infty(\rho)\leq C_p {1\over  \Big(1-{\rho\over r}\Big)^{2\lz+1}  } m_\infty(r)^\alpha m_p(r)^p,\noz
\end{align*}
and hence
\begin{align}\label{sub har inequ 2 ee6}
{m_\infty(\rho)\over H^{1\over p}}\leq C_p {1\over  \Big(1-{\rho\over r}\Big)^{2\lz+1}  } {m_\infty(r)^\alpha\over H^{\alpha\over p}} {m_p(r)^p\over H}.
\end{align}

To continue, we consider the following 3 cases.

\medskip
Case 1: $R>1$.
\medskip

Let $a<1$ and $\rho=r^a$. We now take the logarithm on both side of \eqref{sub har inequ 2 ee6} and integrate on both side from $R$ to $5R$. Then we have 
\begin{align}\label{sub har inequ 2 ee7}
\int_{R}^{5R} \log\Big( {m_\infty(r^a)\over H^{1\over p}}\Big) {dr\over r}&\leq  \int_{R}^{5R} \log\Bigg( {C_p\over  \big(1-{\rho\over r}\big)^{2\lz+1}  } {m_\infty(r)^\alpha\over H^{\alpha\over p}} {m_p(r)^p\over H}\Bigg) {dr\over r}\\
& \leq  \int_{R}^{5R} \log\Bigg( {C_p\over  \big(1-{\rho\over r}\big)^{2\lz+1}  } \Bigg) {dr\over r}\noz\\
&\quad+  \int_{R}^{5R} \log\Bigg( {m_\infty(r)^\alpha\over H^{\alpha\over p}} \Bigg) {dr\over r}\noz\\
&\quad+  \int_{R}^{5R}\log\bigg( {m_p(r)^p\over H}\bigg) {dr\over r}.\noz
\end{align}

Next we claim that
\begin{align}\label{sub har inequ 2 ee8}
\int_{R}^{5R} m_p(r)^p  {dr\over r}\ls  H.
\end{align}

To see this, note that  
 \begin{align*}
H&:=\dfrac{ 1}{\tilde m_\lz(B((0, 0), 5R))}\int_0^{5R}\int_0^\pi| u(r\cos \theta, r\sin \theta)|^p (r\sin \theta)^{2\lz}d\theta\,rdr\\
&={c_\lambda\over R^{2\lz+2}} \int_R^{5R}\int_0^\pi| u(r\cos \theta, r\sin \theta)|^p (\sin \theta)^{2\lz}d\theta\, r^{2\lz+1}dr\\
&\geq {c_\lambda\over R^{2\lz+2}} \int_{R}^{5R}\int_0^\pi| u(r\cos \theta, r\sin \theta)|^p (\sin \theta)^{2\lz}d\theta\, r^{2\lz+1}dr\\
&\gs R^{2\lz+1}{c_\lambda\over R^{2\lz+2}} \int_{R}^{5R}\int_0^\pi| u(r\cos \theta, r\sin \theta)|^p (\sin \theta)^{2\lz}d\theta\, dr\\
&\approx {1\over R} \int_{R}^{5R}\int_0^\pi| u(r\cos \theta, r\sin \theta)|^p (\sin \theta)^{2\lz}d\theta\, dr\\
&\approx  \int_{R}^{5R}\int_0^\pi| u(r\cos \theta, r\sin \theta)|^p (\sin \theta)^{2\lz}d\theta\, {dr\over r},
\end{align*}
which implies the claim \eqref{sub har inequ 2 ee8}.

Then by Jensen's inequality and \eqref{sub har inequ 2 ee8}, we get that
\begin{align}\label{sub har inequ 2 ee9}
\int_{R}^{5R} \log\bigg( {m_p(r)^p\over H}\bigg) {dr\over r}
\ls \log\bigg( \int_{R}^{5R}  {m_p(r)^p\over H} {dr\over r}\bigg) 
\ls 1.
\end{align}

Substituting \eqref{sub har inequ 2 ee9} back into \eqref{sub har inequ 2 ee7}, we obtain that
\begin{align}\label{sub har inequ 2 ee10}
\int_{R}^{5R} \log\Big( {m_\infty(r^a)\over H^{1\over p}}\Big) {dr\over r}
& \leq \wz C_p+\alpha \int_{R}^{5R} \log\Bigg( {m_\infty(r)\over H^{1\over p}} \Bigg) {dr\over r}.
\end{align}

Changing the variable on the left-hand side of \eqref{sub har inequ 2 ee10}, we obtain that
\begin{align}\label{sub har inequ 2 ee10e1}
{1\over a}\int_{R^a}^{(5R)^a} \log\Big( {m_\infty(r)\over H^{1\over p}}\Big) {dr\over r}
& \leq \wz C_p+\alpha \int_{R}^{5R} \log\Bigg( {m_\infty(r)\over H^{1\over p}} \Bigg) {dr\over r},
\end{align}
which gives
\begin{align}\label{sub har inequ 2 ee11}
&{1\over a}\int_{R^a}^{R} \log\Big( {m_\infty(r)\over H^{1\over p}}\Big) {dr\over r} + \Big({1\over a}-\alpha\Big)\int_{R}^{(5R)^a} \log\Big( {m_\infty(r)\over H^{1\over p}}\Big) {dr\over r}\\
&\quad-\alpha \int_{(5R)^a}^{5R} \log\Big( {m_\infty(r)\over H^{1\over p}}\Big) {dr\over r}\noz\\
& \leq \wz C_p.\noz
\end{align}
Now for arbitrary small $\epsilon>0$, by choosing $a$ less than 1 but sufficiently close to 1, we obtain that ${1\over a}-\alpha>0$ and that
\begin{align*}
\alpha \int_{(5R)^a}^{5R} \log\Big( {m_\infty(r)\over H^{1\over p}}\Big) {dr\over r}<\epsilon,
\end{align*}
which implies that
\begin{align*}
\int_{R}^{(5R)^a} \log\Big( {m_\infty(r)\over H^{1\over p}}\Big) {dr\over r} \leq {\wz {\wz C}}_p.
\end{align*}
 Thus, there exists $r_0$ such that 
 $$ m_\infty(r_0)\leq \overline{C}_p. $$

\medskip
Case 2: $5R<1$.
\medskip

Now let $a>1$ and $\rho=r^a$.
Then repeating again the steps in Case 1, we obtain again \eqref{sub har inequ 2 ee11} for $a>1$. Hence,  for arbitrary small $\epsilon>0$, by choosing $a$ greater than 1 but sufficiently close to 1, we obtain that ${1\over a}-\alpha>0$ and that
\begin{align*}
\alpha \int_{(5R)^a}^{5R} \log\Big( {m_\infty(r)\over H^{1\over p}}\Big) {dr\over r}<\epsilon,
\end{align*}
which implies that
\begin{align*}
\int_{R}^{(5R)^a} \log\Big( {m_\infty(r)\over H^{1\over p}}\Big) {dr\over r} \leq {\wz {\wz C}}_p.
\end{align*}
 Thus, there exists $r_0$ such that 
 $$ m_\infty(r_0)\leq \overline{C}_p. $$

%\section{Preliminaries}
%\label{s2}

\medskip
Case 3: $1/5\leq R\leq1$.
\medskip

Now let $a>1$ and $\rho=r^a$.
Then repeating again the steps in Case 1, we obtain again \eqref{sub har inequ 2 ee10e1} for $a>1$. 
 
\begin{align*}
&{1\over a}\int_{R^a}^{R} \log\Big( {m_\infty(r)\over H^{1\over p}}\Big) {dr\over r} + \Big({1\over a}-\alpha\Big)\int_{R^a}^{5R} \log\Big( {m_\infty(r)\over H^{1\over p}}\Big) {dr\over r}-\alpha \int_{5R}^{(5R)^a} \log\Big( {m_\infty(r)\over H^{1\over p}}\Big) {dr\over r}\noz \leq \wz C_p,\noz
\end{align*}
which, again, implies that  there exists $r_0$ such that 
 $$ m_\infty(r_0)\leq \overline{C}_p. $$

Combining the three cases above, we obtain that \eqref{sub har inequ 2 ee2} holds, which implies \eqref{sub har inequ 2 ee1}.
And hence, we obtain that \eqref{sub har inequ 2} holds.

In the end, combining the estimates in Case (i) and Case (ii), we obtain that \eqref{sub har inequ}  holds, which finishes the proof of
Theorem \ref{t-moser inequ}.

\section{Applications}
\label{sec:application}
\setcounter{equation}{0}

As an application, we can obtain a direct proof of equivalent characterizations of the Hardy spaces associated to Bessel operator $\Delta_\lz$ via non-tangential maximal function and radial maximal functions  defined in terms of the Poisson semigroup.

Recall that in \cite{bdt}, they introduced the Hardy spaces associated with $\Delta_\lz$ via the Riesz transforms, radial maximal functions, and showed that they are equivalent.To be more specific,
consider the following spaces (see Definition 1.1 in \cite{bdt}):

\smallskip
(a) $H^1_{\Delta_\lz,Riesz}(\R_+,\dmz):=\{ f\in L^1(\R_+,\dmz):\ \ \riz(f)\in L^1(\R_+,\dmz)  \}$ with the norm 
$$ \|f\|_{H^1_{\Delta_\lz,Riesz}(\R_+,\dmz)}:=\|f\|_{L^1(\R_+,\dmz)}+\|\riz(f)\|_{L^1(\R_+,\dmz)}.  $$

%\smallskip
(b) $H^1_{\Delta_\lz,max}(\R_+,\dmz):=\{ f\in L^1(\R_+,\dmz):\ \ \mathcal{R}(f)\in L^1(\R_+,\dmz)  \}$ with the norm 
$$ \|f\|_{H^1_{\Delta_\lz,max}(\R_+,\dmz)}:=\|f\|_{L^1(\R_+,\dmz)}+\big\|\mathcal{R}(f)\big\|_{L^1(\R_+,\dmz)},  $$
where $$\mathcal{R}(f)(x):=\sup_{t>0}\big|e^{-t\sqrt{\Delta_\lz}}(f)\big|.$$

The following result was proved (\cite{bdt} Theorem 1.7): let $\lz>0$ and $f\in L^1(\R_+,\dmz)$. Then the following assertions are equivalent:

\smallskip
(i) $f\in H^1_{\Delta_\lz,Riesz}(\R_+,\dmz)$;

\smallskip
(ii) $f\in H^1_{\Delta_\lz,max}(\R_+,\dmz)$.

\smallskip
Moreover, the corresponding norms are equivalent.

\medskip

We now consider the Hardy spaces associated to Bessel operator $\Delta_\lz$ via non-tangential maximal function as follows. 
\begin{defn}
Suppose $\lz>0$. Define
$$H^1_{\Delta_\lz,n-max}(\R_+,\dmz):=\Big\{ f\in L^1(\R_+,\dmz):\ \ \mathcal{N}(f)\in L^1(\R_+,\dmz)  \Big\}$$ with the norm 
$$ \|f\|_{H^1_{\Delta_\lz,n-max}(\R_+,\dmz)}:=\|f\|_{L^1(\R_+,\dmz)}+\big\|\mathcal{N}(f)\big\|_{L^1(\R_+,\dmz)}, $$
where $$\mathcal{N}(f)(x):=\sup_{\substack{|x-y|<t\\y\in\R_+}}\big|e^{-t\sqrt{\Delta_\lz}}(f)(y)\big|.$$
\end{defn}
Then, based on our main result, the Moser inequality, we obtain that
\begin{thm}\label{thm appl}
Let $\lz>0$ and $f\in L^1(\R_+,\dmz)$. Then the following assertions are equivalent:

\smallskip
{\rm(ii)} $f\in H^1_{\Delta_\lz,max}(\R_+,\dmz)$;

\smallskip
{\rm(iii)} $f\in H^1_{\Delta_\lz,n-max}(\R_+,\dmz)$.

\smallskip
Moreover, the corresponding norms are equivalent.

\end{thm}

We point out that this theorem is not new, since it is contained in a more general result on spaces of homogeneous type, see for example \cite{GLY,yy,YZ}, where they obtained the equivalence of non-tangential and radial maximal function defined via the approximations to identity, using the auxiliary grand maximal function as a bridge. 

Here in this specific Bessel setting, we can obtain a direct proof by using the Moser inequality of the $\lz$-harmonic functions as we established in Theorem \ref{t-moser inequ}. 

\medskip
\begin{proof}[Proof of Theorem \ref{thm appl}]
Suppose $\lz>0$ and $f\in H^1_{\Delta_\lz,n-max}(\R_+,\dmz)$.  It is obvious 
that for every $x\in \R_+$
$$\mathcal{R}(f)(x)\leq  \mathcal{N}(f)(x),  $$
which implies that
$$ \|f\|_{H^1_{\Delta_\lz,max}(\R_+,\dmz)}\leq  \|f\|_{H^1_{\Delta_\lz,n-max}(\R_+,\dmz)},$$
i.e, we have $f\in H^1_{\Delta_\lz,max}(\R_+,\dmz)$.
Thus, we obtain that
$$  H^1_{\Delta_\lz,n-max}(\R_+,\dmz)\subset H^1_{\Delta_\lz,max}(\R_+,\dmz).  $$

Conversely, Suppose $f\in H^1_{\Delta_\lz,max}(\R_+,\dmz)$. Let
%\begin{equation*}
$u(t,\,x):=e^{-t\sqrt{\Delta_\lz}}(f)(x)=\plz f(x)$.
%\end{equation*}
Then $u(t,\,x)$ is $\lz$-harmonic, i.e., $u(t,\,x)$ satisfies \eqref{bessel laplace equation}.

For any $q\in(0, 1)$, for all $y,t\in\R_+$ with $|y-x|<t$,  from Theorem \ref{t-moser inequ} with $R:=t$, we deduce that
\begin{eqnarray*}
\lf| u(t,\,y)\r|^q
&&\ls  \dfrac1{\tilde m_\lz(B((t, x), 12t))}\int_{B((t,\, x),\,12t)}\lf|u(s,\,z)\r|^q\,z^{2\lz}\,dzds\\
&&\ls  \dfrac1{\tilde m_\lz(B((t, x), 12t))}\int_{B((t,\, x),\,12t)}\mathcal{R}(f)(z)^q\,z^{2\lz}\,dzds\\
&&\ls  \dfrac1{m_\lz(12I)}\int_{12I} \mathcal{R}(f)(z)^q\,z^{2\lz}\,dz\\
&&\ls \cm\lf(\mathcal{R}(f)^q\r)(x),
\end{eqnarray*}
where $I:=I(x, t)$ and $\cm$ is the Hardy--Littlewood maximal function.
 
 This implies that
\begin{equation}\label{nontang max contr by radia max poiwise}
 \mathcal{N}(f)(x)\ls \lf\{\cm\lf[\lf( \mathcal{R}(f)\r)^q\r](x)\r\}^\frac1q.
\end{equation}

 By taking the $L^1$ norm on both side of the inequality above and using the boundedness of the Hardy--Littlewood maximal function, we obtain that
\begin{equation*}
\big\|  \mathcal{N}(f)\big\|_{L^1(\R_+,\dmz(x))}\ls  \lf\|   \mathcal{R}(f)\r\|_{L^1(\R_+,\dmz(x))},  
\end{equation*}
which implies that
$$ \|f\|_{H^1_{\Delta_\lz,n-max}(\R_+,\dmz)}\ls  \|f\|_{H^1_{\Delta_\lz,max}(\R_+,\dmz)},$$
i.e, we have $f\in H^1_{\Delta_\lz,n-max}(\R_+,\dmz)$.
Thus, we obtain that
$$  H^1_{\Delta_\lz,max}(\R_+,\dmz)\subset H^1_{\Delta_\lz,n-max}(\R_+,\dmz).  $$

This completes the proof of Theorem \ref{thm appl}.
\end{proof}

\bigskip
\bigskip
{\bf Acknowledgments:}
X. T. Duong is supported by ARC DP 140100649. Z.H. Guo is supported by Monash University New Staff Grant.
J. Li is supported by ARC DP 160100153 and Macquarie University New Staff Grant.
D. Yang is supported by the NNSF of China (Grant No. 11571289) and the State Scholarship Fund of China (No. 201406315078).

\end{document}